\documentclass[10pt,a4paper]{article}

\usepackage{latexsym}                   
\usepackage{epsfig,color}                     
\usepackage{graphicx}
\usepackage{amssymb, amsmath, amsthm}   
\usepackage{enumerate}                  
\usepackage{mathrsfs}                   
\usepackage{geometry}                                                                          
\usepackage{verbatim}                                                                           
\usepackage{lineno}
\usepackage{setspace}

\newtheorem{thm}{Theorem}[section]

\newtheorem{cor}[thm]{Corollary}
\newtheorem{lem}[thm]{Lemma}

\theoremstyle{definition}

\theoremstyle{definition}
\newtheorem{rem}[thm]{Remark}
\theoremstyle{definition}

\newcommand{\bE}{{\mathbb E}}

\newcommand{\bP}{{\mathbb P}}

\newcommand{\cT}{{\mathcal T}}

\newcommand{\citet}[1]{\cite{#1}}
\newcommand{\citep}[1]{\cite{#1}}

\begin{document}

\thispagestyle{plain}
\title{
Distribution of  branch lengths and phylogenetic diversity under homogeneous speciation models}
\author{Tanja Stadler$^1$ \& Mike Steel$^2$\\
\date{\today}
{\small $^1$ Institut f\"{u}r Integrative Biologie, ETH Z\"{u}rich } \\ 
{\small Universit\"{a}tsstr. 16, 8092 Z\"{u}rich, Switzerland} \\ 
{\small Phone +41 44 632 45 48, Fax +41 44 632 12 71, tanja.stadler@env.ethz.ch}\\
{\small $^2$ Biomathematics Research Centre} \\ 
{\small University of Canterbury, Christchurch, New Zealand} \\ 
{\small mike.steel@canterbury.ac.nz}\\
}
\maketitle
\doublespacing

\begin{abstract}
The constant rate birth--death process is a popular null model for speciation and extinction. If one removes extinct  and non-sampled lineages, this process induces  `reconstructed trees' which describe the relationship between extant lineages.  We derive the probability density of the length of a randomly chosen pendant edge in a reconstructed tree. For the special case of a pure-birth process with complete sampling, we also provide the probability density of the length of an interior edge, of the length of an edge descending from the root, and of the diversity (which is the sum of all edge lengths). We show that the results depend on whether the reconstructed trees are conditioned on the number of leaves, the age, or both.
\end{abstract}


\newpage

\section{Introduction}
The constant rate birth--death process is a widely-used null model for speciation and extinction \cite{Mooers1997, Nee2001}.
This model has been used to test the hypothesis of constant macroevolutionary rates and to quantify the rates of speciation \cite{Paradis1998,Pybus2000,Ricklefs2007,Stadler2011PNAS}. Despite its wide use, the process continues to reveal new and sometimes unexpected results: even the simple Yule model leads to a curious property highlighted in a recent paper \cite{Steel2010,Mooers2011}:  the expected length of a randomly chosen edge in a Yule tree is half of the expected waiting time until a speciation event occurs. In order to attribute such ``surprises'' in empirical data to the null model instead of trying to find further explanations, the null models need to be well understood. In this paper, we fully characterize the lengths of pendant edges in birth--death trees on extant species (so-called reconstructed trees), and improve our understanding of interior edge lengths in Yule trees.

We will first explain the concept of a reconstructed tree which was originally introduced in  \cite{Nee1994} (see also Fig. \ref{fig_one}). The birth--death process starts with a single species at time $x_0$ before the present. At all times until the present, each species has a constant rate $\hat{\lambda}$ of speciation and a constant rate $\hat{\mu}$ of extinction (with $0\leq \hat{\mu} \leq \hat{\lambda}$). Such a process induces a birth--death tree. At the present, each extant species is sampled with probability $f$. 
Throughout this paper, we prune extinct and non-sampled species in the birth--death tree, i.e. we consider the  birth--death tree which is induced by the sampled extant species. The tree without the extinct and non-sampled species is called the reconstructed tree, as empirical data typically  infers this reconstructed tree (unless fossil information is included).

We consider three different scenarios for stopping the process (i.e. defining the present): 
\begin{itemize}
\item Scenario (i):  we condition the process on having $n$ extant sampled species, or 
\item Scenario (ii):  we condition the process on having $n$ extant sampled species and age $x_1$ for the  most recent common ancestor of the extant species, or 
\item Scenario (iii): we condition the process on having age $x_1$ for the  most recent common ancestor of the extant species.
\end{itemize}

As the start of the process (the time of origin $x_0$ or the time of most recent common ancestor $x_1$) is a parameter of the birth--death model, we have to assume a prior distribution for the time of origin  when not conditioning the reconstructed trees on its age (Scenario (i)). We make the common assumption that the first species originated at any time $x_0$ in the past with uniform probability \cite{AlPo2005}. This is also called an improper prior on $(0,\infty)$. Conditioning the resulting reconstructed tree to have $n$ extant species yields a proper distribution for the time of origin \cite{Gernhard2008JTB}. 

A special case of the birth--death process is the Yule model
 \cite{Yule1924}, which is obtained by setting $\hat{\mu}=0$ and $f=1$.  
Under the Yule model, Scenario (i) is equivalent to stopping the process just before the $(n+1)$-th speciation event \cite{Stadler2010SystBiol}.

\begin{figure}[ht]
\begin{center}
\resizebox{10cm}{!}{
\includegraphics{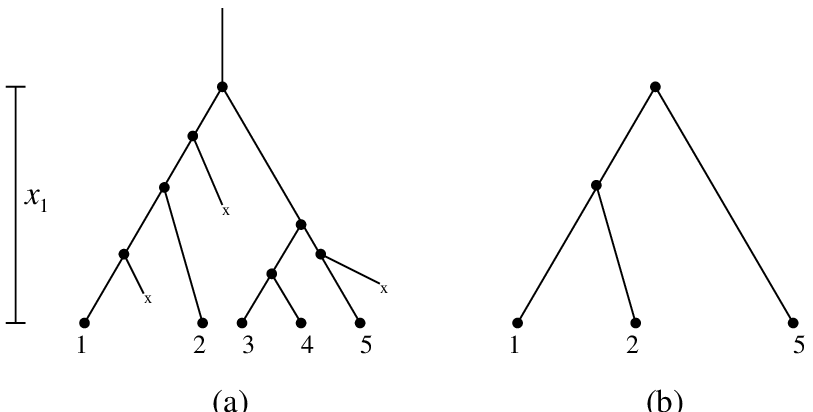}
}
\caption{(a): A birth--death tree starting with a single lineage, and with extinct lineages ending in `x'. The time from the extant lineages to their most recent common
ancestor is $x_1$. (b): The `reconstructed tree' obtained by deleting the initial single lineage and all the extinct and two non-sampled extant lineages (namely the two labelled 3 and 4).}
\label{fig_one}
\end{center}
\end{figure}

\begin{rem} \label{RemTrans}
In \cite{Stadler2009JTB}, 
it is shown that a birth--death process with parameters $\hat{\lambda}, \hat{\mu}, f$ under Scenario (i) or (ii) induces the same distribution on reconstructed trees as a birth--death process with parameters $\lambda, \mu$ and complete extant species sampling, where: $$\lambda=f \hat{\lambda}, \mu = \hat{\mu} - \hat{\lambda}(1-f).$$ Thus, under Scenario (i) and (ii), we will state all birth--death model results as functions of the transformed variables $\lambda, \mu$ and complete sampling. We will further establish in this paper that the transformation also holds under Scenario (iii).
\end{rem}

In the following, we derive the probability density of the length of a randomly chosen  pendent edge in a reconstructed tree  generated by a constant rate birth--death process under Scenarios (i), (ii), and (iii). 
For the special case of a Yule model (with rate $\lambda$) under Scenario (i), we find that a randomly selected pendant  edge  (or interior edge) has an exponentially distributed length with parameter $2 \lambda$ (Corollary \ref{pendyulen} and Theorem~\ref{ThmIntn}). This result generalizes \cite{Steel2010}, where the expected length of a pendant  edge was calculated to be $1/(2 \lambda)$. 
For the Yule model (under Scenarios (i)-(iii)), we also derive  the probability density of the length of an edge descending from the root and the sum of all edges.

\section{Preliminaries}

We first present some notation and a preliminary result  that will be useful later.
A reconstructed tree on $n$ extant species has $n-1$ interior vertices at times $x_1> \ldots > x_{n-1}$ in the past (see Fig.~\ref{fig_two}).  We call the speciation event at time $x_k$ the $k$-th speciation event.
We say that a leaf $x$ of a reconstructed tree is {\em adjacent to} the $k$-th speciation event if $(v,x)$ is an arc of the tree, where $v$ is the vertex that corresponds to the $k$-th speciation event. For example, in Fig.~\ref{fig_two}
the left-most leaf of the tree is adacent to the second speciation event.   In \cite{Stadler2008MB}, the following result was established; we will provide a shorter and more direct proof here.

\begin{figure}[ht]
\begin{center}
\resizebox{10cm}{!}{
\includegraphics{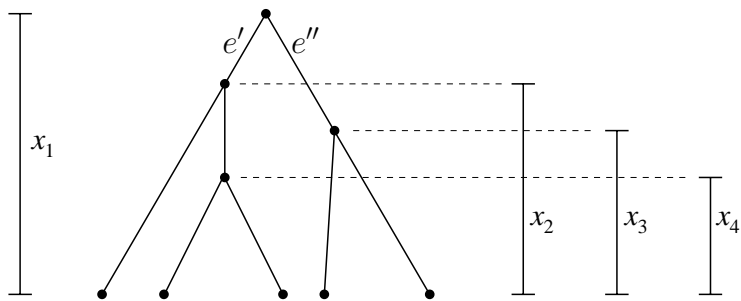}
}
\caption{For $k=1, \ldots, n-1$, $x_k$ denotes the time from the present into the past when the $k$--th speciation event occurred in a reconstructed birth-death tree.  The two edges incident with the root $e'$ and $e''$  (considered in Section \ref{rootsec}) are also shown, with $e'$ the shorter of the two. }
\label{fig_two}
\end{center}
\end{figure}

\begin{thm}
For a reconstructed Yule or  birth-death tree on $n\geq 2$ extant species, the probability under Scenario (i) or (ii) that a randomly-selected leaf is adjacent to the $k$-th speciation event  is:
\begin{eqnarray}
v(k) = \frac{2k}{n(n-1)}. \label{pk}
\end{eqnarray}
\end{thm}
\begin{proof}
Our proof is based on the equivalence of the probability distributions on tree topologies under three models -- the Yule process, the reconstructed birth--death tree and the coalescent tree (if one ignores branch lengths) \cite{Aldous2001}.  Observe that leaf $L$ is adjacent to the $k$-th speciation event vertex in the
Yule tree if and only if $L$ fails to coalesce for the first $n-k-1$ coalescence events, but does so on
the next coalecence event (when  $k+1$ points are available to coalesce and $k$ of these possible ${k+1 \choose 2}$ pairs involve $L$). The probability of this, under the coalescent process, is precisely:
$$\prod_{j=1}^{n-k-1}\left[1-\frac{n-j}{{n-j+1 \choose 2}}\right] \frac{k}{{k+1 \choose 2}}$$
Now, expansion and simplication (cancellation) of this product yields $v(k) = \frac{2k}{n(n-1)}$.
\end{proof}
The following expressions will be useful later. 
\noindent
For $0 \leq \mu \leq \lambda$, we define,
\begin{equation}
\label{p0eq}
p_0(s):=
\begin{cases}
 \frac{(1-e^{-(\lambda-\mu)s})}{\lambda-\mu e^{-(\lambda-\mu)s}}, & \mbox{ if } \mu< \lambda; \\ \frac{ s}{1+\lambda s},\, & \mbox{ if  }\mu = \lambda;
\end{cases}
\end{equation}
and 
\begin{equation}
\label{p1eq}
p_1(s):=
\begin{cases}
 \frac{ (\lambda-\mu)^2 e^{-(\lambda-\mu)s}}{(\lambda - \mu e^{-(\lambda-\mu)s})^2}, & \mbox{ if } \mu< \lambda; \\  \frac{1}{(1+\lambda s)^2},\, & \mbox{ if  }\mu = \lambda.
\end{cases}
\end{equation}

The significance of these quantities is that the probability of a lineage producing $0$ (resp. $1$) offspring after time $s$ is $\mu p_0(s)$ (resp.  $p_1(s)$) \cite{Kendall1949}.
Note that, for Yule trees, we have: $$p_0(s)= (1-e^{-\lambda s})/\lambda; \mbox{ and }  p_1(s) = e^{-\lambda s}.$$

\section{Length of a pendant edge in  birth--death and Yule trees}
In this section, we calculate the probability density function (pdf) of the length of a random pendant edge in a reconstructed tree. For  Scenario (i), we use this to derive the expected length of such an edge; moreover we calculate the pdf of the length of a random {\em interior} edge for a Yule model.
For Scenario (ii), we use the pdf to calculate the expected length of a random pendant edge.  For Scenario (iii), we additionally establish a general transformation equivalence between models (Theorem~\ref{transformthm}).

\subsection*{(i) Conditioning on $n$}

For pendant edges, the following result was established in \cite{Mooers2011}.

\begin{thm} \label{Thmpendn}
The length of a randomly selected pendant edge in a reconstructed birth--death tree on $n$ extant species has probability density function:
$$f_p(s|n) = 2 \lambda  p_1(s) (1-\lambda p_0(s)),$$
and expectation,
$$\bE[p|n] =\frac{ \mu + (\lambda - \mu) \log(1-\mu / \lambda)}{\mu^2} .$$
\end{thm}

Here we note the following direct consequence of this result for the Yule model (obtained by setting $\mu=0$).

\begin{cor} \label{pendyulen}
The length of a randomly picked pendant edge in a tree on $n$ leaves under the Yule model with rate $\lambda$ is exponentially distributed with rate $2 \lambda$.
\end{cor}

We now establish a corresponding result for a randomly selected interior edge in a Yule tree (again under Scenario (i)).

\begin{thm} \label{ThmIntn}
The length of a randomly selected interior edge in a Yule (rate $\lambda$) tree on $n$ leaves is exponentially distributed with rate $2 \lambda$.
\end{thm}

\begin{proof}
We will establish this theorem by induction on $n$. For $n=3$, we have one  interior edge. The waiting time between the first speciation event (yielding two species) and the second speciation event (yielding three species) is  the length of the interior edge.  This waiting time is the time until the first of the two species speciates. As each species has a speciation rate $\lambda$, this waiting time is an exponential distribution with rate $2 \lambda$.

Now assume we established the theorem for $n=k$.
For establishing the theorem for $k+1$, we first note that the Yule tree on $k+1$ species with the uniform prior for the time of origin corresponds to  a Yule process forward in time which is stopped just before the $(k+1)$-th speciation event \cite{Stadler2010SystBiol}.
Now cut off the tree on $k+1$ leaves $\cT_{k+1}$ at time $x_k$, yielding a Yule tree on $k$ leaves $\cT_k$. A randomly selected nterior edge $\cT_k$ has an exponentially distributed length with rate $2 \lambda$ (induction assumption). Each interior edge in $\cT_k$ is also an interior edge in $\cT_{k+1}$. Additionally, in $\cT_{k+1}$, one of the pendant edges becomes an interior edge (the edge which speciates at time $x_k$). As  each  randomly selected pendant edge in $\cT_k$ has an exponentially distributed length with rate $2 \lambda$ (Corollary \ref{pendyulen}), the new interior edge in $\cT_{k+1}$ also has an exponentially distributed length with rate $2 \lambda$. This yields that a randomly picked interior edge in $\cT_{k+1}$ has an exponentially distributed length with rate $2 \lambda$, and thereby establishes the induction step.
\end{proof}

\subsection*{(ii) Conditioning on $n$ and $x_1$}
In order to derive the edge length distribution, we need the following lemma.
\begin{lem}
The probability density function for the time of the $k$-th speciation event in a reconstructed birth--death tree, conditional on having $n\geq 2$ extant species and on the first speciation event being at time $x_1$ is:
\begin{equation*}
f_{n,k} (s|n,x_1) = (n-2) {n-3 \choose k-2} G(s|x_1)^{n-k-1} 
(1-G(s|x_1))^{k-2} g(s|x_1), k=2,\ldots n-1, 
\end{equation*}
where
\begin{equation*}
g(s|x_1) = \frac{p_1(s)}{p_0(x_1)}, \qquad G(s|x_1) = \frac{p_0(s)}{p_0(x_1)} ,
\end{equation*}
with $p_0$ and $p_1$ the functions described in Eqns. (\ref{p0eq}) and (\ref{p1eq}). 
\end{lem}
\begin{proof}
In \cite{Gernhard2008JTB}, we established that the speciation times  $x_2> \ldots > x_{n-1} $ are the order statistics of $n-2$ i.i.d. random variables $s_2,\ldots,s_{n-2}$ with probability density $g(s|x_1)$.
For the distribution,  by integrating  $g(s|x_1)$ with respect to $s$ we obtain the function $G(s|x_1)$.
Now, $x_{k+1}$ is the $k$-th order statistic (with the first order statistic being the largest and the $(n-2)$-th order statistic being the smallest value), thus we have the probability density function for $x_k$ (e.g. \cite{DehlingStochastik}):
$$f_{n,k+1} (s|t) = (n-1-k) {n-2 \choose n-1-k} G(s|t)^{n-k-2} 
(1-G(s|t))^{k-1} g(s|t). $$
Equivalently:
$$f_{n,k+1} (s|t) = (n-2) {n-3 \choose k-1} G(s|t)^{n-k-2} 
(1-G(s|t))^{k-1} g(s|t). $$
\end{proof}

\begin{thm} \label{Thmfpnx1}
The length of a randomly  selected pendant edge in a reconstructed birth--death tree on $n$ extant species and age $x_1$  has the probability density function:
$$f_p(s|n,x_1) = 2  \frac{ (n-2)}{n(n-1)}  \frac{p_1(s)}{p_0(x_1)} \left( (n-1)  - (n-3) \frac{p_0(s)}{p_0(x_1)} \right) $$
for $s<x_1$ and:
$$f_p(x_1|n,x_1) = \frac{2}{n(n-1)} \delta(0),$$
where $\delta$ is the Dirac delta function.
\end{thm}

\begin{proof}
The proof is similar to the proof of Theorem \ref{Thmpendn} in \cite{Mooers2011}.
For $s=x_1$, the pendant edge must be attached to the root which has probability $2/(n(n-1))$ (Eqn. (\ref{pk})). The length of the pendant edge is $x_1$ with mass equal to $1$ which is formalized using the Dirac delta function. 

For $s<x_1$, we have: \begin{eqnarray*}
f_p(s|n,x_1) &=& \sum_{k=2}^{n-1} v(k) f_{n,k}(s|n,x_1)\\
&=& 2  \frac{ (n-2)}{n(n-1)} \sum_{k=2}^{n-1}  k {n-3 \choose k-2} G(s|x_1)^{n-k-1} (1-G(s|x_1))^{k-2} g(s|x_1)\\
&=& 2  \frac{ (n-2)}{n(n-1)} g(s|x_1) \sum_{k=0}^{n-3}  (k+2) {n-3 \choose k} G(s|x_1)^{n-k-3} (1-G(s|x_1))^{k} \\
&=& 2  \frac{ (n-2)}{n(n-1)} g(s|x_1) \left( 2+      \sum_{k=1}^{n-3} k {n-3 \choose k} G(s|x_1)^{n-k-3} (1-G(s|x_1))^{k} \right) \\
&=& 2  \frac{ (n-2)}{n(n-1)} g(s|x_1) \left( 2+      \sum_{k=0}^{n-4} (n-3) {n-4 \choose k} G(s|x_1)^{n-k-4} (1-G(s|x_1))^{k+1} \right) \\
&=& 2  \frac{ (n-2)}{n(n-1)} g(s|x_1) \left( 2+    (n-3) (1-G(s|x_1) \right)\\
&=& 2  \frac{ (n-2)}{n(n-1)} g(s|x_1) \left( (n-1)  - (n-3)G(s|x_1) \right) ,
\end{eqnarray*}
which establishes the theorem.
\end{proof}
Evaluating the first moment integral of $f_p(s|n,x_1)$ yields,
\begin{cor} \label{CorExpnx1}
For $0<\mu<\lambda$, the expected length of a pendant edge is:
\begin{eqnarray*}
 \bE[p|n,x_1]&=&\frac{1}{(n-1) n}\left(2 x_1+\frac{(n-2) }{p_0(x_1) \left(1-e^{-(\lambda-\mu) x_1}\right)} C \right),
 \end{eqnarray*}
 where
 \begin{eqnarray*}
C &=& 
\frac{ (n-3)p_0(x_1)}{ \lambda \mu}  (\lambda-\mu e^{-(\lambda-\mu) x_1} )-
\\ & & \frac {x_1 p_1(x_1)\left(\lambda-e^{-(\lambda-\mu) x_1} \mu\right)}{\lambda^2} \left(\frac{4}   {\lambda-\mu} \left(\lambda-e^{-(\lambda-\mu) x_1} \mu\right)  -e^{-(\lambda- \mu) x_1} (n+1)\right)-
\\ & & \frac{\log (1-\mu p_0(x_1))}{(\lambda \mu)^2} \left( e^{-(\lambda-\mu) x_1} \mu (-4 \mu-(n+1)(\lambda -\mu ))- \lambda (-4\lambda+(n+1)(\lambda -\mu ))\right).
\end{eqnarray*}
For $\mu=\lambda$, 
$$  \bE[p|n,x_1]= \frac{1}{(n-1) n}\left(2{x_1}+\frac{(n-2) }{x_1 p_0(x_1) \lambda^2} \left( (n-7){x_1}+ (n+1) p_0(x_1) +\frac{ 6-2n +4 {\lambda} {x_1} }{\lambda} \log [1+{\lambda} x_1] \right)\right).$$
\end{cor}

Setting $\mu=0$ in Corollary \ref{CorExpnx1} yields Theorem 2 of \cite{Mooers2011}.

\subsection*{(iii) Conditioning on $x_1$}

\begin{lem} \label{Lempn}
The probability of a reconstructed birth--death tree with the first speciation event being at time $x_1$ having $n \geq 2$ extant descendants is,
$$p_n(x_1) = (n-1) \frac{p_1(x_1)^2 (\lambda p_0(x_1))^{n-2}}{(1-\mu p_0(x_1))^2}.$$
In particular, the probability $p_n(x_1) $ can be written as a function of the transformed parameters $\lambda,\mu$ instead of $\hat{\lambda},\hat{\mu},f$.
\end{lem}
\begin{proof}
The probability of a single individual having $n$ extant and sampled offspring after time $x_1$, given that it has at least one offspring is \cite{Stadler2010JTB}:
$$ \hat{p}_n(x_1)= \frac{f (\hat{\lambda}- \hat{\mu})^2 e^{-(\hat{\lambda}- \hat{\mu}) x_1}}{ (f \hat{\lambda}+(\hat{\lambda}(1-f) -\hat{\mu}) e^{-(\hat{\lambda}- \hat{\mu}) x_1})^2  }  \left( \frac{f \hat{\lambda}(1- e^{-(\hat{\lambda}- \hat{\mu}) x_1}) }{f \hat{\lambda}+(\hat{\lambda}(1-f) -\hat{\mu}) e^{-(\hat{\lambda}- \hat{\mu}) x_1}}\right)^{n-1}  \frac{f \hat{\lambda}+(\hat{\lambda}(1-f) -\hat{\mu}) e^{-(\hat{\lambda}- \hat{\mu}) x_1}}{f (\hat{\lambda}- \hat{\mu})}$$
Using the transformation of the parameters $\hat{\lambda}$ and $\hat{\mu}$ to $\lambda$ and $\mu$ as given in Remark \ref{RemTrans}, we can express $\hat{p}_n(x_1)$ as a function of only $\lambda,\mu,n,x_1$:
$$\hat{p}_n(x_1)= \frac{p_1(x_1) (\lambda p_0(x_1))^{n-1}}{1-\mu p_0(x_1)}.$$
This leads to:
\begin{eqnarray*}
p_n(x_1) &=& \sum_{k=1}^{n-1}  \frac{p_1(x_1) (\lambda p_0(x_1))^{k-1}}{1-\mu p_0(x_1)}  \frac{p_1(x_1) (\lambda p_0(x_1))^{n-k-1}}{1-\mu p_0(x_1)}\\
&=& (n-1) \frac{p_1(x_1)^2 (\lambda p_0(x_1))^{n-2}}{(1-\mu p_0(x_1))^2}.
\end{eqnarray*}
\end{proof}

\begin{thm}
The length of a randomly picked pendant edge in a reconstructed  birth--death tree of age $x_1$ has probability density function for $s<x_1$:
\begin{eqnarray*}
f_p(s|x_1) &=& 2\frac{p_1(s)p_1(x_1)^2}{p_0(x_1)(1-\mu p_0(x_1))^2}\left( h(1,x_1) -  \frac{p_0(s)}{p_0(x_1)} h(3,x_1)\right),
\end{eqnarray*}
where
\begin{eqnarray*}
 h(k,x_1) &:=&\frac{(k+1)\lambda p_0(x_1) - k}{(1- \lambda p_0(x_1))^2}  - \frac{2k \log[1-\lambda p_0(x_1)] }{(\lambda p_0(x_1))^2}- \frac{2k }{\lambda p_0(x_1)},
\end{eqnarray*}
 and for $s=x_1$,
$$f_p(x_1|x_1) = -2(\log(1-\lambda p_0(x_1)) +\lambda p_0(x_1)) \left( \frac{p_1(x_1)}{\lambda p_0(x_1) (1-\mu p_0(x_1))}\right)^2 \delta(0),$$
where $\delta$ is the Dirac delta function.
 In particular, the probability $f_p(s|x_1)$ can be written as a function of the transformed parameters $\lambda,\mu$ instead of $\hat{\lambda},\hat{\mu},f$.
\end{thm}

\begin{proof}
Using Theorem \ref{Thmfpnx1} and Lemma \ref{Lempn}, for $s<x_1$, we have (using the transformed parameters $\lambda,\mu$ instead of $\hat{\lambda},\hat{\mu},f$):
\begin{eqnarray*}
f_p(s|x_1) &=& \sum_{n=3}^\infty f_p(s|n,x_1) p(n|x_1)\\
&=& \sum_{n=3}^\infty  2  \frac{ (n-2)}{n(n-1)}  \frac{p_1(s)}{p_0(x_1)} \left( (n-1)  - (n-3) \frac{p_0(s)}{p_0(x_1)} \right)  (n-1) \frac{p_1(x_1)^2 (\lambda p_0(x_1))^{n-2}}{(1-\mu p_0(x_1))^2}\\
&=& 2\frac{p_1(s)p_1(x_1)^2}{p_0(x_1)(1-\mu p_0(x_1))^2} \sum_{n=3}^\infty   \frac{ (n-2)}{n}   \left( (n-1)  - (n-3) \frac{p_0(s)}{p_0(x_1)} \right)   (\lambda p_0(x_1))^{n-2}
\end{eqnarray*}
We also have:
\begin{eqnarray*}
 h(k,x_1):= &=& \sum_{n=3}^\infty   \frac{ n-2}{n}  (n-k)  (\lambda p_0(x_1))^{n-2} \\&=&  \sum_{n=3}^\infty   (n-2)  (\lambda p_0(x_1))^{n-2} 
 - \sum_{n=3}^\infty  k \frac{n-2}{n}  (\lambda p_0(x_1))^{n-2}\\
 &=& \frac{\lambda p_0(x_1)}{(1- \lambda p_0(x_1))^2} -  \frac{k}{1-\lambda p_0(x_1)} +k + \frac{2k}{(\lambda p_0(x_1))^2}  \sum_{n=3}^\infty \frac{(\lambda p_0(x_1))^{n}}{n}\\
&=& \frac{(k+1)\lambda p_0(x_1) - k}{(1- \lambda p_0(x_1))^2} +k + \frac{2k \left( - \log[1-\lambda p_0(x_1)] -\lambda p_0(x_1) -\frac{(\lambda p_0(x_1))^2}{2}\right)}{(\lambda p_0(x_1))^2}\\
&=& \frac{(k+1)\lambda p_0(x_1) - k}{(1- \lambda p_0(x_1))^2}  - \frac{2k \log[1-\lambda p_0(x_1)] }{(\lambda p_0(x_1))^2}- \frac{2k }{\lambda p_0(x_1)}.
\end{eqnarray*}
Overall, this yields:
\begin{eqnarray*}
f_p(s|x_1) &=& 2\frac{p_1(s)p_1(x_1)^2}{p_0(x_1)(1-\mu p_0(x_1))^2}\left( h(1,x_1) -  \frac{p_0(s)}{p_0(x_1)} h(3,x_1)\right).
\end{eqnarray*}
For $s=x_1$, we obtain $f_p(x_1|x_1)$ in an analogue way, using  $ f_p(x_1|x_1) = \sum_{n=2}^\infty f_p(x_1|n,x_1) p(n|x_1)$.
\end{proof}
Evaluating the first moment integral of $f_p(s|x_1)$, or evaluating $\sum_{n=2}^\infty \bE[p|n,x_1] p(n)$ with $  \bE[p|n,x_1] $ from 
Corollary~\ref{CorExpnx1}, provides an analytic expression for the expected length of a randomly chosen pendant edge.

The following theorem establishes a general result under Scenario (iii), namely that a birth--death model with incomplete sampling can be transformed into a birth--death model with complete sampling when considering reconstructed trees.
\begin{thm}
\label{transformthm}
Let $T$ be a reconstructed tree conditioned on $x_1$. The probability density of  $T$ having speciation times $x_2,\ldots,x_{n-1}$ with parameters $\hat{\lambda},\hat{\mu},f$ equals the probability density of $T$ having speciation times $x_2,\ldots,x_{n-1}$   with parameters $\lambda,\mu$ (where $\lambda=f \hat{\lambda}, \mu = \hat{\mu} - \hat{\lambda}(1-f)$, as above) under complete sampling.
\end{thm}
\begin{proof}
The probability density of $x$ is provided in  \cite{Stadler2010JTB}, Theorem 3.8, with $m=0$,
$$f(x|x_1)= \left(\frac{q_1(x_1)}{1-q_0(x_1)}\right)^2 \cdot \prod_{i=2}^{n-1} \hat{\lambda} q_1(x_i),$$
where
\begin{eqnarray*}
q_0(x_1) &=&   1-  \frac{f (\hat{\lambda}- \hat{\mu}) }{ f \hat{\lambda}+(\hat{\lambda}(1-f) -\hat{\mu}) e^{-(\hat{\lambda}- \hat{\mu}) x_1}  },\\
q_1(x_1) &=& \frac{f (\hat{\lambda}- \hat{\mu})^2 e^{-(\hat{\lambda}- \hat{\mu}) x_1}}{ (f \hat{\lambda}+(\hat{\lambda}(1-f) -\hat{\mu}) e^{-(\hat{\lambda}- \hat{\mu}) x_1})^2  }.
\end{eqnarray*}
The transformation $\lambda=f \hat{\lambda}, \mu = \hat{\mu} - \hat{\lambda}(1-f)$ now establishes the theorem.
 \end{proof}

\section{Length of a root edge in a Yule tree}
\label{rootsec}

A reconstructed tree has two edges descending from the root, and we denote these as $e'$ and $e''$
where we may assume that $e'$ is shorter than $e''$ (see Fig.~\ref{fig_two}).  By a {\em root edge} we mean the selection of $e'$ or $e''$ with equal probability. We will calculate the length of the root edge in a Yule tree under Scenario (i)-(iii).

We will show that, under Scenario (i), a root edge is longer than a randomly chosen interior edge (which has exponential distribution with parameter $2 \lambda$, Theorem \ref{ThmIntn}).

\subsection*{(i) Conditioning on $n$}

Let $X_i$ be an exponentially distributed random variable with parameter $i \lambda$. Then for $k \geq 2$,  $I_k=\sum_{i=2}^k X_i$ is the hypo-exponential distribution with probability density:
\begin{eqnarray*}
f_{I_k}(t) &=& \sum_{i=2}^k  \lambda i e^{-\lambda i t} \prod_{j=2,j\neq i}^k \frac{j}{j-i},
\end{eqnarray*}
which can be transformed to
\begin{eqnarray}
f_{I_k}(t) &=&  \sum_{i=2}^k  \lambda  e^{-\lambda i t} (-1)^{i-2} \frac{k!}{(k-i)!(i-2)!} \notag \\
&=& k(k-1) \sum_{i=2}^k  \lambda  (-e^{-\lambda  t})^i  {k-2 \choose i-2}.\label{EqnIk}
\end{eqnarray}

\begin{thm}
The length $L$ of one of the two edges descending from the root (picked uniformly at random) has probability density function:
$$f_L(t|n) =  \lambda  e^{-\lambda  t} \left(  1-(1-e^{-\lambda  t})^{n-2}(1-ne^{-\lambda  t}) \right).$$
\end{thm}

\begin{proof}
In a Yule tree with $n$ extant species, the shorter edge $e'$ of the two edges descending from the root has a length $L'$ that is exponentially distributed with  parameter $2 \lambda$.
The length $L''$ of the longer edge descending from the root, $e''$, is calculated as follows. First note that the waiting time, $X_i$, between the $i-1$-th and $i$-th speciation event is an exponential distribution with parameter $i \lambda$. Let $e''$ terminate at the $k$-th speciation event in the tree. The length of $A$ is then $I_k$ with probability density given in Eqn. (\ref{EqnIk}).
The probability that $A$ terminates at the $k$-th speciation event ($2<k<n$) is:
$$p_k= \frac{1}{k} \prod_{j=3}^{k-1} \left(1-\frac{1}{j}\right)= 1/{k \choose 2},$$
and the probability that $A$ does not terminate until the present is,
$$p_n=\prod_{j=3}^{n-1} \left(1-\frac{1}{j}\right) = 2/(n-1).$$
The length of $e''$ is:
$$L''= \sum_{k=3}^n I_k p_k,$$
and thus the density function of $L''$ is,
\begin{eqnarray*}
f_{L''}(t|n) &=&  \sum_{k=3}^n f_{I_k}(t) p_k\\
&=& 2 \sum_{k=3}^{n-1}
\sum_{i=2}^k  \lambda  (-e^{-\lambda  t})^i  {k-2 \choose i-2} + 2 n \sum_{i=2}^n  \lambda (-e^{-\lambda  t})^i  {n-2 \choose i-2}\\
&=& 2 
\sum_{i=3}^{n-1} \sum_{k=i}^{n-1} \lambda  (-e^{-\lambda  t})^i  {k-2 \choose i-2} +  2 \sum_{k=3}^{n-1} \lambda  e^{- 2 \lambda  t} +  2 n \sum_{i=2}^n  \lambda  (-e^{-\lambda  t})^i {n-2 \choose i-2}
\\
&=& 2 
\sum_{i=3}^{n-1} \lambda (-e^{-\lambda  t})^i  {n-2 \choose i-1} + 2 (n-3) \lambda  e^{- 2 \lambda  t} +  2 n \sum_{i=2}^n  \lambda  (-e^{-\lambda  t})^i  {n-2 \choose i-2}\\
&=& 2 
\sum_{i=3}^{n-1} \lambda (-e^{-\lambda  t})^i  \left( {n-2 \choose i-1}+ n {n-2 \choose i-2} \right)   + 2(2n-3) \lambda  e^{- 2 \lambda  t}  +  2 n  \lambda  (-e^{-\lambda  t})^n  \\
&=& 2 
\sum_{i=3}^{n-1} \lambda  (-e^{-\lambda  t})^i   i {n-1 \choose i-1}   + 2(2n-3) \lambda  e^{- 2 \lambda  t}  +  2 n  \lambda  (-e^{-\lambda  t})^n  \\
&=& 2 
\lambda \sum_{i=3}^{n-1}  i (-e^{-\lambda  t})^i    {n-1 \choose i-1}  + 2(2n-3) \lambda  e^{- 2 \lambda  t}  +  2 n  \lambda  (-e^{-\lambda  t})^n \\
&=&  2 \lambda  e^{-\lambda  t} \left(  1-e^{-\lambda  t}-(1-e^{-\lambda  t})^{n-2}(1-n e^{-\lambda  t})   \right).
\end{eqnarray*}
The density of $L$ is,
\begin{eqnarray*}
f_L(t|n) &=& f_{L'}(t)/2 + f_{L''}(t)/2\\
&=&  \lambda  e^{-\lambda  t} \left(  1-(1-e^{-\lambda  t})^{n-2}(1-ne^{-\lambda  t})
\right),
\end{eqnarray*}
which establishes the theorem.
\end{proof}
Calculating the first moment of $L$ yields:
\begin{cor}
\label{root_n_cor}
The expected length $L$ of one of the two edges descending from the root (selected at random) is $\bE[L|n]=\frac{1}{\lambda}(1-\frac{1}{n})$.
\end{cor}
Comparing this result with Corollary \ref{pendyulen} and Theorem \ref{ThmIntn} we see that for a Yule tree conditioned on having $n$ species, 
the expected length of one of the two edges descending from the root, selected at random, is (asymptotically) {\em twice} the length of a randomly selected edge (or of a randomly selected pendant edge, or of a randomly selected interior edge).

\subsection*{(iii) Conditioning on $x_1$}
Before considering Scenario (ii), we first consider Scenario (iii), i.e. the time since the root is $x_1$. 
In a Yule tree with the root having age $x_1$, select one of the two edges incident with the root uniformly at random (e.g. by a fair coin toss), and let $L$ denote its length (up to time $x_1$).  Then $L$ has a discontinuous distribution:
$$\bP(L >l|x_1) =
\begin{cases}
e^{-\lambda l}, \mbox{ for } 0<\l< x_1;\\
0, \mbox{ for } l\geq x_1;
\end{cases}
$$
which implies that $\bE[L|x_1] = \frac{1}{\lambda} (1-e^{-\lambda x_1})$, and so, in particular:
$$\bE[L|x_1] = \frac{1}{\lambda} - o(1),$$ where $o(1)$ is a term that goes to zero exponentially fast with $x_1$.

Note that the distribution of $L$ should not be confused with the truncated exponential distribution arising from the conditional probability
$\bP(L>l| L \leq t)$, since this is the probability that a speciation event occurs on this lineage before or at time $x_1$.

\subsection*{(ii) Conditioning on $n$ and $x_1$}

We first  consider a Yule tree starting from a single lineage at time $0$ and grown for time $t$. Let $K_t$ be the number of leaves at time $t$, and consider
the length $I$ of the initial edge up to time $t$. Let $\bP(I>l | K_t=k)$ denote the probability that $I$ is greater than $l$, conditional on the event that $K_t=k$.

\begin{lem}
\label{ilem}
$$\bP(I>l |K_t=k) = 
\begin{cases}
\alpha^{k-1}, \mbox{ if } l <  t;\\
0, \mbox{ if } l\geq t;
\end{cases}
$$
where $\alpha = \frac{1-e^{-\lambda(t-l)}}{1-e^{-\lambda t}}.$ 
\end{lem}

\begin{proof}

Let $f(s|K_t=k)$ be the density of $I$ conditional on $K_t=k$. 
By Bayes' formula:
\begin{equation}
\label{bayes}
f(s|K_t=k)=\frac{\mathbb{P}(K_t=k|I=s)f(s)}{\mathbb{P}(K_t=k)},
\end{equation}
where $$f(s)= \lambda e^{-\lambda s}$$ is the (unconditional) first branch length.  
The unconditional distribution of $K_t$ is $$\mathbb{P}(K_t=k)=(1-e^{-\lambda t})^{k-1}e^{-\lambda t}.$$ 
Moreover,  it is an easy exercise to show that, for $s< t$, $$\mathbb{P}(K_t=k|I=s)=(k-1)(1-e^{-\lambda(t-s)})^{k-2}e^{-2\lambda(t-s)}.$$
Combining these three expressions into (\ref{bayes}) gives:
\begin{equation}
\label{fs}
f(s|K_t=k) =\frac{(k-1)\lambda e^{-\lambda (t-s)}(1-e^{-\lambda (t-s)})^{k-2}}{(1-e^{-\lambda t})^{k-1}}.
\end{equation}
Now,
\begin{equation}
\mathbb{P}(I >l |  K_t=k)=\int_l^t\ f(s|K_t=k)ds.
\end{equation}\\
and so substituting (\ref{fs}),  applying the substitution $u=1-e^{-\lambda t} e^{\lambda s}$ and rearranging, one obtains:
\\
\begin{equation}
\mathbb{P}(I > l | K_t=k)=\left[ \frac{k-1}{(1-e^{-\lambda t})^{k-1}}\right]\cdot\int_0^{1-e^{-\lambda t}e^{\lambda l}}u^{k-2} du,
\end{equation}
which gives the value $\alpha^{k-1}$ for $l < t$.
\end{proof}

Consider now a Yule tree $T$ and suppose we condition on both $n$ and $x_1$.  Select uniformly at random one of the edges of  $T$  that are incident with the root,  let $L$ denote its length up to time $x_1$, and  let $\bP(L>l|n,x_1)$ denote the probability that $L>l$, conditional on $T$ having $n$ leaves at time $x_1$. 

\begin{thm}
\begin{equation}
\label{basicp}
\bP(L>l|n, x_1) = \begin{cases}
\frac{1}{n-1}\cdot \left(\frac{1-\alpha^{n-1}}{1-\alpha}\right), &  \mbox{\rm for }  l\leq x_1;\\
0, & \mbox{\rm  for } l > x_1.
\end{cases}
\end{equation}
\end{thm}
\begin{proof}
It is a fundamental property of the Yule model that the number of leaves beneath the selected edge has a uniform distribution between $1$ and $n-1$. Thus, if
$l \leq x_1$ then: $$\bP(L>l|n, x_1) = \frac{1}{n-1} \sum_{k=1}^{n-1} \bP(I>l|K_{x_1}=k)$$ and the result now  follows directly by  Lemma~\ref{ilem}. 
\end{proof}

Let us now replace $\lambda$ by its maximum likelihood (ML) estimate $\lambda_{ML}=\ln(\frac{n}{2})/x_1$ and evaluate the expression for $\bP(L>t|n, x_1)$ as $n \rightarrow \infty$.  
Note that if we let $\lambda = \ln(\frac{n}{2})/t$ then $$\alpha = \frac{1-\frac{2}{n}e^{\lambda l}}{1-\frac{2}{n}},$$
and so:
$$\alpha^{n-1} = \frac{(1-\frac{2}{n}e^{\lambda l})^{n-1}}{(1-\frac{2}{n})^{n-1}}\sim \frac{\exp(-2e^{\lambda l})}{e^{-2}}, $$
where
$\lambda = \lambda_{ML} =\ln( \frac{n}{2})/t$ and where $\sim$ denotes asymptotic equivalence as $n \rightarrow \infty$. Moreover,
$$\frac{1}{n-1} \cdot \frac{1}{1-\alpha} = \frac{n}{2(n-1)}\frac{1-\frac{2}{n}}{e^{\lambda l}-1} \sim \frac{1}{2(e^{\lambda l}-1)}.$$
So, from (\ref{basicp}), we  obtain:
$$\bP(L>l|n, x_1)  \sim \frac {1-e^{-w}}{w}$$ where $w = 2(e^{\lambda l} - 1)$ and where $\sim$ denotes asymptotic equivalence as $n \rightarrow \infty$.

 To determine the expectation of $L$ conditional on $n, x_1$ we simply integrate
this expression from $L=0$ to $L=x_1$ with respect to $l$ (using the well known identity that $\bE[X] = \int_0^\infty \bP(X>x) dx$ for a non-negative continuous random variable $X$). Noting that $dl = \frac{dw}{\lambda(2+w)}$, we obtain the following result which exhibits a different limit to the value $\frac{1}{2}\lambda$ for a randomly selected edge of $T$ or the limiting value $\frac{1}{\lambda}$ described above when we just  condition on $x_1$. 
\begin{cor}
The expected length of one of the two randomly selected root edges,  conditional on $n$ and $x_1$, and with $\lambda$ set equal to its ML value, converges to $\frac{c}{\lambda}$ as $n \rightarrow \infty$ where:
$$c = \int_{0}^{\infty} \frac{1-e^{-x}}{x(2+x)}dx = 0.8158...$$
\end{cor}
It is instructive to compare this result with Corollary \ref{root_n_cor}. 

\section{Diversity in a Yule tree}
In this section, we calculate the sum of all edge lengths in a Yule tree which is also called the diversity.

\subsection*{(i) Conditioning on $n$}

\begin{thm} \label{ThmDivnx1}
The sum $D$ of all branches in a Yule tree with $n\geq 2$ leaves has a gamma distribution with density function:
$$f_D(d|n)=\frac{\lambda^{n-1} e^{-\lambda d} d^{n-2}}{(n-2)!}.$$
In particular, $D$ has mean $\bE[D|n] = \frac{n-1}{\lambda}$ and variance ${\rm Var}[D|n] = \frac{n-1}{\lambda^2}$ and 
is asymptotically normally distributed. 
\end{thm}
\begin{proof}
The sum of all edge lengths in a Yule tree with $n$ leaves is $D=\sum_{j=2}^{n} j X_j$ where $X_j$ has an exponential distribution with parameter $\lambda j$. 
Note that $jX_j$ has an exponential distribution with parameter $\frac{1}{j} \cdot \lambda j=\lambda$, and so $D$ is a sum of $n-1$ independent exponential random variables, each having parameter $\lambda$, and so $D$ has the claimed gamma distribution.

\end{proof}

\subsection*{(ii) Conditioning on $n$ and $x_1$}
As shown in \cite{Gernhard2008JTB}, a Yule tree of age $x_1$ with $n\geq 3$ leaves, the $n-2$ speciation events $S_2,\ldots,S_{n-1}$ descending from the root are i.i.d. random variables with density:
$$g(s) = \frac{\lambda e^{-\lambda s}}{1-e^{-\lambda x_1}}.$$
The sum of all branch lengths is the sum of the $n-2$ speciation times and $2x_1$ (accounting for the two branches descending from the root). 

\begin{thm} \label{ThmDivnx1}
The sum $D$ of all branches in a Yule tree of age $x_1$ and $n$ leaves has moment generating function: 
$$M_D(s|n, x_1) =  e^{2x_1s}\cdot \left( \frac{\lambda(1-e^{(s-\lambda )x_1})}{(\lambda-s)(1-e^{-\lambda x_1})} \right)^{n-2}.$$ 
\end{thm}
\begin{proof}
Let $C$  denote the sum of the $n-2$ speciation times $S_2,\ldots,S_{n-1}$. Then 
$C$ is a convolution of $n-2$ random variables with probability density $g(s) = \frac{\lambda e^{-\lambda s}}{1-e^{-\lambda x_1}}$. If $M_C(s|n, x_1)$ denotes the moment generating function for $C$, then:
\begin{eqnarray*}
M_C(s| n,x_1) &=& \bE[e^{Cs}] = \bE[e^{\sum_{j=2}^{n-1} S_j s }] = \prod_{j=2}^{n-1}  \bE[e^{ S_j s }] \\
&=& \left( \int_{-\infty}^\infty e^{sx} g(x) dx \right)^{n-2}\\
&=& \left( \frac{\lambda}{{1-e^{-\lambda x_1}}} \int_{0}^{x_1} e^{sx-\lambda x}  dx \right)^{n-2}\\
&=& \left( \frac{\lambda(1-e^{(s-\lambda )x_1})}{(\lambda-s)(1-e^{-\lambda x_1})} \right)^{n-2}\\
\end{eqnarray*}
Now $D=2x_1+C$ and so $M_D(s|n, x_1) = e^{2x_1s}M_C(s|n, x_1)$, which leads directly to the expression claimed.
\end{proof}

\subsection*{(iii) Conditioning on $x_1$}

The sum $D$  of all branches in a Yule tree of age $x_1$ has probability density function,
$$f_D(d |x_1) = \sum_{n=2}^\infty  p_n(x_1) f_D(d|n,x_1),$$
where $f_D(d|n,x_1)$ is the density conditional on $n$ and $x_1$.  
In particular, noting that $\bE[D|n, x_1]= \frac{d}{ds}M_D(s|n, x_1)_{|s=0}$, we have:
 $$\bE[D|x_1] =  \sum_{n=2}^\infty  p_n(x_1) \frac{d}{ds}M_D(s|n, x_1)_{|s=0},$$
 which gives: 
$$\bE[D|x_1] = \frac{2}{\lambda}(e^{\lambda x_1} -1 ),$$
a result that was derived via a different argument in \cite{Steel2010}.

\section{Conclusion}
In this paper, we derive the probability density and expectation for the length of a randomly picked pendant edge of a reconstructed birth--death tree with random sampling of extant individuals.
We investigate this under three scenarios in which the resulting reconstructed trees induced by the birth--death process are finite: we either consider trees (i) with a fixed number of leaves, (ii) with a fixed number of leaves and a fixed age of the most recent common ancestor of the extant sampled species, or (iii) with a fixed age of the most recent common ancestor of the extant sampled species.

We first noted that under our three Scenarios (i)-(iii), the original process with sampling (parameters $\hat{\lambda},\hat{\mu},f$) can be transformed into a birth--death process and complete sampling through $\lambda=f \hat{\lambda}, \mu = \hat{\mu} - \hat{\lambda}(1-f).$  Hence, we state all results as functions of the transformed parameters $\lambda,\mu$.

For the Yule model we further determine under our three Scenarios (i)-(iii) the probability density for the length of an edge descending from the root and the sum of all edges, which is also known as the diversity.
In particular,  we  show  that the pendant as well as the interior edge lengths under Scenario (i) are exponentially distributed with parameter $2 \lambda$ while an edge descending from the root is asymptotically twice as long compared to a randomly chosen edge. Furthermore, the diversity follows a gamma distribution.

Knowledge of the branch lengths and diversity in a phylogenetic tree is important for conservation strategies. When present-day species become extinct, short pendant edges cause little diversity loss while long pendant edges cause severe diversity loss. In previous work, the expected loss of diversity was considered under constant rate birth-death models with mild extinction \cite{Mooers2011}, and under constant rate birth-death models with severe extinction, namely $\lambda=\mu$ \cite{Nee1997} (note that \cite{Nee1997} derived the results under the coalescent with constant population size; in expectation, such coalescent trees equal birth-death trees with $\lambda=\mu$ \cite{Gernhard2008BMB}). 
Characterizing the full  distribution of the loss of diversity remains an open task. Our results yielding the branch length and diversity distribution might be a first step towards characterizing the distribution of diversity loss (rather than only the expectation). Knowledge of the  distribution will help understand the stochastic variability of diversity loss.
An asymptotic normal law for diversity loss was recently established by \cite{Faller2008}, but this was established only for certain deterministic classes trees, rather than for trees generated by a stochastic process. 


For future work, it would  be interesting to compare  the branch lengths and diversity (loss) under the constant rate birth-death process to models under which the rates of speciation and extinction may vary. In particular, speciation and extinction rates may  dependent on the  number of species that are extant at that moment  (i.e. density-dependent speciation) \cite{Rabosky2008}, or on time (i.e. environmental-dependent speciation) \cite{Nee1994,Stadler2011PNAS},  or on a particular trait (i.e. trait-dependent speciation) \cite{Maddison2007}.

\subsection{Acknowledgments} MS thanks the Royal Society of New Zealand (James Cook Fellowship) and the Allan Wilson Centre for Molecular Ecology and Evolution for funding.   TS thanks the ETH Zurich for funding and the Royal Society of New Zealand (Marsden Fund)  for travel support. 
\bibliographystyle{abbrv}
\bibliography{bibliography2}

\end{document}